\documentclass[11pt,oneside,english,reqno]{amsart}
\usepackage[T1]{fontenc}
\usepackage[latin9]{inputenc}
\usepackage{color}
\usepackage{babel}
\usepackage{amsthm}
\usepackage{amstext}
\usepackage{amssymb}
\usepackage{setspace}
\usepackage[unicode=true,pdfusetitle,
 bookmarks=true,bookmarksnumbered=false,bookmarksopen=false,
 breaklinks=false,pdfborder={0 0 0},backref=false,colorlinks=true]
 {hyperref}
\usepackage[usenames,dvipsnames]{xcolor}
\usepackage{nicefrac}
\usepackage{mathrsfs}

\makeatletter
  \theoremstyle{plain}
  \newtheorem*{thm*}{\protect\theoremname}
\theoremstyle{plain}
\newtheorem{thm}{\protect\theoremname}
  \theoremstyle{plain}
  \newtheorem{lem}[thm]{\protect\lemmaname}

\usepackage{calrsfs}
\usepackage{graphicx}
\usepackage{color}
\usepackage{perpage}
\usepackage{upquote}
\usepackage{bbm}
\usepackage{enumitem}

\MakePerPage{footnote}
\gdef\SetFigFontNFSS#1#2#3#4#5{} 
\gdef\SetFigFont#1#2#3#4#5{} 
\def\clap#1{\hbox to 0pt{\hss#1\hss}}

\DeclareMathOperator{\len}{len}

\DeclareMathOperator{\saw}{saw}
\DeclareMathOperator{\rw}{rw}
\newcommand{\Bcal}{\mathscr{B}}
\newcommand{\Ccal}{\mathscr{C}}
\newcommand{\Lcal}{\mathscr{L}}
\newcommand{\Gcal}{\mathscr{G}}

\definecolor{myblue}{rgb}{0.09,0.32,0.44} 
\hypersetup{pdfborder={0 0 0},pdfborderstyle={},colorlinks=true,linkcolor=myblue,citecolor=myblue,urlcolor=blue}

\theoremstyle{definition}
\newtheorem*{qst*}{Question}
\newtheorem*{rmrks*}{Remarks}
\newtheorem*{rem*}{Remark}

\newlength{\tempindent} 
\newcommand{\lazyenum}{
\setlength{\tempindent}{\parindent} 
\begin{enumerate}[leftmargin=0cm,itemindent=0.7cm,labelwidth=\itemindent,labelsep=0cm,align=left,label=\arabic*)]
\setlength{\parskip}{\smallskipamount}
\setlength{\parindent}{\tempindent}
}
\allowdisplaybreaks

\makeatother

  \providecommand{\lemmaname}{Lemma}
  \providecommand{\theoremname}{Theorem}
\providecommand{\theoremname}{Theorem}

\newcommand{\Z}{{\mathbb Z}}
\newcommand{\R}{{\mathbb R}}

\newcommand{\eqan}[1]{\begin{align}#1\end{align}}
\newcommand{\eqn}[1]{\begin{equation}#1\end{equation}}
\newcommand{\nn}{\nonumber}
\newcommand{\sss}{\scriptscriptstyle}
\def\1{{\mathchoice {1\mskip-4mu\mathrm l}      
{1\mskip-4mu\mathrm l}
{1\mskip-4.5mu\mathrm l} {1\mskip-5mu\mathrm l}}}

\newcommand{\e}{{\mathrm{e}}}

\begin{document}

\title{Lace expansion for dummies}
\author{Erwin Bolthausen}
\address{Institut f\"ur Mathematik, Universit\"at Z\"urich, Winterthurerstrasse 190, 8057 Z\"urich, Switzerland
and Research Institute for Mathematical Sciences, Kyoto University, Kyoto 606-8502, Japan.
e-mail: \tt{eb@math.uzh.ch}}

\author{Remco van der Hofstad}
\address{Department of Mathematics and
    Computer Science, Eindhoven University of Technology, P.O.\ Box 513,
    5600 MB Eindhoven, The Netherlands.
e-mail: \tt{r.w.v.d.hofstad@tue.nl}}

\author{Gady Kozma}

\address{
Weizmann Institute, Rehovot, 76100, Israel. e-mail: \tt{gady.kozma@weizmann.ac.il}
}



\begin{abstract}
We show Green's function asymptotic upper bound for the two-point function of weakly self-avoiding walk in $d>4$, revisiting a classic problem. Our proof relies on Banach algebras to analyse the lace-expansion fixed point equation and is simpler than previous approaches in that it avoids Fourier transforms.
\end{abstract}

\maketitle

\section{Introduction}

The lace expansion made its debut in 1985 with a proof by Brydges and Spencer that weakly self-avoiding
walk (WSAW) has ``Gaussian behaviour'' in dimensions $5$ and above \cite{BS85}.
It proved to be useful way beyond its initial application, primarily
in work by Hara and Slade. The technique was applied to 
percolation \cite{HS90}, lattice trees and animals \cite{HS90b}, the contact process \cite{S01}, the Ising model \cite{S07} and $\varphi^4$ \cite{S15}. Further, it was extended to finite
graphs \cite{BCHSS05} and to long-range models \cite{CS15}. Despite
all this progress, weakly self-avoiding walk remains the simplest
example to which the technique applies: lace expansion is a ``perturbative''
technique and it requires a small parameter.
Weakly self-avoiding walk has such a small parameter naturally
built-in, while for most models, the small parameter is more hidden. Consequently, it was used as a test bed for several new
techniques, for example in \cite{HHS98}, where the lace expansion was 
analysed using induction in time,  and \cite{BR01} where a Banach fixed point theorem was used. 
Interestingly, neither of these papers uses the so-called bootstrap analysis
introduced in \cite{S87}. In our opinion, the bootstrap analysis is
the most important simplification to lace expansion, replacing the 
difficult ``moving single pole'' analysis of \cite{BS85}. The bootstrap 
analysis applies to generating functions such as the WSAW 
Green's function, while \cite{BR01,HHS98} 
prove results for WSAW with a fixed number of steps instead.
Green's function asymptotics in $x$-space as derived here were proved previously in technically more challenging settings 
in \cite{HHS03} for spread-out models and in \cite{H08} for nearest-neighbour settings. Brydges and Spencer \cite{BS85} prove 
Gaussian limit laws for the end-to-end displacement for WSAW after $n$ steps.

Our starting point was also an attempt to generalise lace expansion,
rather than to simplify it. We wished to apply it to problems on Cayley
graphs of non-commutative groups. Most of the existing approaches rely
heavily on the Fourier transform, which is of course no longer available
in this new setting. The approach of \cite{BR01,ABR13}, though, turned
out to be applicable. We realized that it can be simplified and generalised
by working in an appropriate Banach algebra. 


In this paper, we expose our Banach-algebra approach in the simplest
possible setting: weakly self-avoiding walk on $\mathbb{Z}^{d}$,
with the result being an upper bound on the critical Green's function.
We repeat that related results have been proved previously, our proof is novel.

\section{Precise definitions and statement of the theorem}

For a nearest-neighbour path $\gamma:\{0,\dotsc,n\}\to\mathbb{Z}^{d}$ and a $\beta\in[0,1]$,
we define its weight by
	\begin{equation}
	W(\gamma)=W^{\beta}(\gamma)=(1-\beta)^{|\{0\le s<t\le n:\gamma(s)=\gamma(t)\}|},
	\label{eq:defW}
	\end{equation}
i.e., the path is ``penalized'' by $1-\beta$ for every self-intersection of $\gamma$. We define the weakly self-avoiding walk Green's function to be
	\[
	G_{\lambda}^{\saw}(x)=G_{\lambda}^{\beta,\saw}(x)
	=\sum_{\gamma \colon 0\to x}\lambda^{\len(\gamma)}W^{\beta}(\gamma),
	\]
where the notation $\gamma\colon 0\to x$ means that $\gamma$ is some path
starting at $0$ and ending at $x$, while $\len(\gamma)$ is the number
of edges of $\gamma$ (rather than vertices). We define $\lambda_{c}$
to be the critical value for the finiteness of the spatial sum of $G_{\lambda}^{\saw}$, i.e.,
	\[
	\lambda_{c}=\sup\Big\{\lambda\colon \sum_{x\in\mathbb{Z}^{d}}G_{\lambda}^{\saw}(x)<\infty\Big\}.
	\]
Finally denote by $G^{\rw}(x)$ the (critical) Green's function of
simple random walk (SRW) on $\mathbb{Z}^{d}$, i.e.,
	\[
	G^{\rw}(x)=\sum_{n=0}^{\infty}p_{n}(x)
	\]
where $p_{n}(x)$ is the probability that simple random walk on $\mathbb{Z}^{d}$
starting from $0$ is at $x$ at time $n$. When $d>2$ the sum converges
and $G^{\rw}(x)=(a+o(1))|x|^{2-d}$ as $|x|\to\infty$ with $a>0$. See e.g.\ \cite{U98}. The result 
is that the WSAW  Green's function is bounded by the SRW Green's function for $d>4$:

\begin{thm*}[Green's function upper bound]

Let $d>4$. Then there exists a $\beta_{0}$ such that for all $\beta<\beta_{0}$,
$\beta$-weakly self avoiding walk satisfies
	\[
	G_{\lambda_{c}}^{\saw}(x)\le 2G^{\rw}(x)\qquad\forall x\in\mathbb{Z}^{d}.
	\]
\end{thm*}

\begin{rem*}
We will also show a lower bound,
$G_{\lambda_c}^{\saw}\ge\frac{1}{2}G^{\rw}$, and further that
	\eqn{
	\label{Green-asymp}
	G_{\lambda_{c}}^{\saw}(x)=(1+O(\beta))G^{\rw}(x)\qquad\textrm{as
          $\beta\to 0$, uniformly in $x\in\mathbb{Z}^{d}$}.
	}
See the remarks on page \pageref{page:rmrks}. We find the upper
bound to be the more interesting and we prefer to focus on it.
\end{rem*}

The remainder of this paper is devoted to the proof of this theorem.

\section{Proof}
\label{sec-pf}
For $\mu\in \R$ we denote by $\Delta^{\rw}_\mu$ the following function
	\eqn{
	\label{srw-delta}
	\Delta_{\mu}^{\rw}(x)=\begin{cases}
	1 & x=0,\\
	-\mu & x\mbox{ is a neighbour of }0,\\
	0 & \mbox{otherwise}.
	\end{cases}
	}
We say that a function $f\colon \mathbb{Z}^{d}\to\mathbb{R}$ is ``symmetric
to coordinate permutations and flipping'' if for any $\sigma\in S_{d}$
(the group of permutations on $d$ elements) and for any $\epsilon_{1},\dotsc,\epsilon_{d}\in\{\pm1\}$,
	\[
	f(x_{1},\dotsc,x_{d})=f(\epsilon_{1}x_{\sigma(1)},\dotsc,\epsilon_{d}x_{\sigma(d)}).
	\]
We further write $f*g$ for the convolution of two functions $f,g\colon \Z^d \mapsto {\mathbb R}$ and $\delta_0$ for Kronecker's delta function. 

\begin{lem}[Lace expansion analysis]
\label{lem:lace}Let $d>4$. Then there exists a $\beta_{0}$ such
that for all $\beta<\beta_{0}$ and for all $\lambda<\lambda_{c}$
the following holds. If $G_{\lambda}^{\saw}(x)\le3G^{\rw}(x)$ for
all $x\in\mathbb{Z}^{d}$, then there exists a function $\Delta_{\lambda}^{\saw}\colon \mathbb{Z}^{d}\to\mathbb{R}$
such that $G_{\lambda}^{\saw}*\Delta_{\lambda}^{\saw}=\delta_{0}$
and such that 
\begin{enumerate}
\item $\Delta_{\lambda}^{\saw}$ is symmetric to coordinate permutations
and flipping;
\item $\sum_{x}\Delta_{\lambda}^{\saw}(x)\ge0;$
\item There exists some $\lambda'\in[0,\frac{1}{2d}]$ such that 
	\[
	|\Delta_{\lambda}^{\saw}(x)-\Delta_{\lambda'}^{\rw}(x)|\le C\beta|x|^{-d-4}.
	\]
\end{enumerate}
\end{lem}

A somewhat abusive convention we adopt here and below is that $|x|^{-\alpha}=1$
when $x=0$, so condition (3) in fact implies that $|\Delta^{\saw}(0)-\Delta^{\rw}(0)|\le C\beta$.
$C$ and $c$ are used for constants that depend only on the dimension.
 Let us remark that in fact we simply take $\lambda'=\min(\lambda,\frac 1{2d})$, though we will not use this fact. Another remark worth making is that in (3) we will in fact prove, $|\Delta^{\saw}(x)-\Delta^{\rw}(x)|\le C\beta|x|^{6-3d}$ which is of course stronger than the stated estimate when $d\ge 5$. However, it will be convenient to formulate the lemma as above.

We remark that it is tempting to think about $\Delta^{\saw}$ as a
generator of some random walk (with killing), but it is missing one
important property of a generator: it is not true that $\Delta^{\saw}(x)<0$
for all $x\ne0$. This means that a lot of deconvolution techniques 
for random walk generators are inapplicable. The next lemma is the
required deconvolution:

\begin{lem}[Deconvolution]
\label{lem:solve}Let $d>2$. Then there exists $\beta_{0}$ such
that for all $\beta<\beta_{0}$ and for any $\Delta\colon \mathbb{Z}^{d}\to\mathbb{R}$
satisfying conditions (1)-(3) of Lemma \ref{lem:lace}, there exists
a function $G$ such that $G*\Delta=\delta_{0}$ and $|G(x)|\le 2G^{\rw}(x)$.
\end{lem}
We postpone the proof of both lemmas and first show how they imply
the theorem:

\begin{proof}
[Proof of the theorem given Lemmas \ref{lem:lace} and \ref{lem:solve}]
Fix $\beta$ to be some value sufficiently small so that both Lemma \ref{lem:lace}
and Lemma \ref{lem:solve} hold with this value of $\beta$. The following argument, known as a bootstrap argument, goes back to Slade \cite{S87}.
Define
	\[
	f(\lambda)=\sup_{x\in\mathbb{Z}^{d}}\frac{G_{\lambda}^{\saw}(x)}{G^{\rw}(x)}.
	\]
We first examine $f(0)$. $G_{0}^{\saw}=\delta_{0}$ and of course
$G^{\rw}\ge\delta_{0}$ so $f(0)\le1$. Next we note that $f$ is
continuous in the interval $[0,\lambda_{c})$. Indeed, $\lambda_{c}$
is the radius of convergence of $\sum_{x}G_{\lambda}^{\saw}(x)$ and
hence (lower bounds) the radius of convergence of $G_{\lambda}^{\saw}(x)$
for all $x$. Hence each term $G_{\lambda}^{\saw}(x)/G^{\rw}(x)$
is continuous on our interval. On the other hand, because the sum
defining $G_{\lambda}^{\saw}(x)$ contains only paths of length at
least $|x|$, it also decays exponentially in $x$, uniformly on $[0,\lambda]$,
for all $\lambda<\lambda_{c}$. This means that on any $[0,\lambda]$ with $\lambda<\lambda_c$,
$f$ can be written as the supremum of a finite collection of continuous
functions, and hence is continuous. Since $\lambda$ can be taken
arbitrarily close to $\lambda_{c}$, $f$ is continuous on $[0,\lambda_{c})$.

We now claim that it is not possible that $f(\lambda )\in (2,3]$
for any $\lambda <\lambda _{c}$. Indeed, if $f(\lambda )\leq 3$ then $%
G_{\lambda }^{\saw}(x)\leq 3G^{\rw}(x)$ for all $x$ and the condition of
Lemma \ref{lem:lace} is satisfied. We use Lemma \ref{lem:lace} to find some $%
\Delta _{\lambda }^{\saw}$ with $G_{\lambda }^{\saw}\ast \Delta _{\lambda }^{%
\saw}=\delta _{0}$ satisfying conditions (1)-(3), and then Lemma \ref%
{lem:solve} to find some $G$ such that $G\ast \Delta _{\lambda }^{\saw%
}=\delta _{0}$ and $G(x)\leq 2G^{\rw}(x)$. We now
claim that $G_{\lambda }^{\saw}=G$. Indeed, both functions are in $\ell ^{2}(%
\mathbb{Z}^{d})$ --- $G_{\lambda }^{\saw}$ by assumption and $G$ by the
conclusion of Lemma \ref{lem:solve} --- and so is $\Delta _{\lambda }^{\saw}$
by condition (3). In $\ell ^{2}$, deconvolution can be performed by Fourier
transform and hence is unique. We get that $G_{\lambda }^{\saw}(x)=G(x)\leq
2 G^{\rw}(x)$ so $f(\lambda )\leq 2$. We conclude that $f(\lambda )\not\in (2,3]$ for any $\lambda
<\lambda _{c}$.

Now, if $f$ is continuous, starts below $1$ and cannot traverse
the interval $(2,3]$, then it must be that $f(\lambda)\le2$ for
all $\lambda<\lambda_{c}$, i.e., $G_{\lambda}^{\saw}(x)\le2G^{\rw}(x)$
for all $x$ and all $\lambda<\lambda_{c}$. Finally, by monotone
convergence, $G_{\lambda_{c}}^{\saw}(x)=\lim_{\lambda\nearrow\lambda_{c}}G_{\lambda}^{\saw}(x)$,
so that also $G_{\lambda_{c}}^{\saw}(x)\leq 2G^{\rw}(x)$ for all $x\in \Z^d$.
\end{proof}
\medskip

We move to the proof of Lemmas \ref{lem:lace} and \ref{lem:solve}.
Lemma \ref{lem:lace} essentially relies on the same lace expansion argument as performed by Brydges and Spencer
\cite{BS85} --- we include the proof for completeness, but we will
be a little brief. 
Lemma \ref{lem:solve} is the new ingredient of our paper.
\begin{proof}[Proof of Lemma \ref{lem:lace}]
We follow \cite[Appendix A]{HHS98} closely for the derivation of the lace expansion, and \cite{HHS03} for the analysis of the coefficients arising in it.
As $\beta$ and $\lambda$ are fixed,
let us remove them from the notation and denote our functions by $G^{\saw}$
and $\Delta^{\saw}$. We start by finding a formula for $\Delta^{\saw}$
(or rather, a representation as an infinite sum). Recall the weight
$W(\gamma)$ defined in (\ref{eq:defW}). We define
	\[
	U_{st}(\gamma)=\begin{cases}
	0 & \text{when~}\gamma(s)\ne\gamma(t),\\
	-\beta &\text{when~}\gamma(s)=\gamma(t),
	\end{cases}\qquad\forall0\le s<t\le\len(\gamma),
	\]
so that
	\[
	W(\gamma)=\prod_{s<t}(1+U_{st}(\gamma)).
	\]
Given an interval $I = [a,b]$ of integers with $0 \leq a \leq b$,
we refer to a pair $\{ s, t\}$ ($s<t$) of elements of $I$ as an {\em edge}.
To abbreviate the notation, we write $st$ for $\{ s,t \}$.
A set of edges is called a {\em graph}. A graph $\Gamma$ on $[a,b]$ is said to
be {\em connected} if both $a$ and $b$ are endpoints of edges in $\Gamma$ and if,
in addition, for any $c \in (a,b)$ there is an edge $st \in \Gamma$
such that $s < c < t$ (note that this is unrelated to the usual
definition of graph connectivity).
The set of all graphs on $[a,b]$ is denoted $\Bcal[a,b]$, and the subset
consisting of all connected graphs is denoted $\Gcal [a,b]$.

For integers $0 \leq a < b$, define
	\eqn{
    	K[a,b](\gamma) = \prod_{a \leq s < t \leq b} ( 1 + U_{st}(\gamma) ),
	\qquad
	\text{so that}
	\qquad
	\label{Gsaw}
    	G^{\saw}(x)=\sum_{\gamma \colon 0 \rightarrow x}\lambda^{\len(\gamma)}K[0,\len(\gamma)](\gamma),
	}
where the sum is over all simple random walk paths from $0$ to $x$. Expanding the product in the definition of $K[a,b](\gamma)$, we get
	\eqn{
	\label{Kdef}
    	K [a,b](\gamma)  = \sum_{\Gamma \in \Bcal [a,b] }
    	\prod_{st \in \Gamma}
    	U_ {st}(\gamma).
	}
For $0 \leq a < b$ we define an analogous quantity, in which the sum over
graphs is restricted to connected graphs, namely,
	\eqn{
	\label{Jconn}
    	J[a,b](\gamma) = \sum_{\Gamma \in \Gcal [a,b] }
    	\prod_{st \in \Gamma}
    	U_{st}(\gamma).
	}
We claim that, for $n\ge 1$,
	\eqn{
	\label{KJiden}
    	K[0,n] = K [1,n] + \sum_{m=2}^{n} J [0,m]K [m,n],
	}
To see this, we note from \eqref{Kdef} that the contribution to $K[0,n]$
from all graphs $\Gamma$ for which $0$ is not in an edge is exactly
$K [1,n]$. To resum the contribution from the remaining graphs, we proceed as follows.
When $\Gamma$ does contain an edge ending at $0$, we let $m(\Gamma)>0$
denote the smallest number that is not crossed by an edge, i.e., there is no $st\in \Gamma$ such that $s<m(\Gamma)<t$. We lose
nothing by taking $m \geq 2$, since $U_{a,a+1}=0$ for all $a$.
Resummation over graphs on $[m,n]$ and \eqref{Jconn} proves \eqref{KJiden}.

Let us now define the key quantities in the lace expansion, which is
	\eqn{
	\label{Pi}
    	\Pi(x) =   \sum_{\gamma\colon 0 \rightarrow x}\lambda^{\len(\gamma)}J [0,\len(\gamma)](\gamma)
	}
and 
	\begin{equation}
	\Delta^{\saw}(x):=\Delta^{\rw}_\lambda(x) -\Pi(x).
	\label{eq:defDsaw}
	\end{equation}	
The key to the proof of Lemma \ref{lem:lace} is the estimate
\begin{equation}\label{eq:Pi absolutely}
\sum_{n=0}^\infty \bigg|\sum_{\substack{\gamma:0\to x \\ \len\gamma=n}}\lambda^n J[0,n](\gamma)\bigg|<C\beta|x|^{6-3d}
\end{equation}
which of course implies as a consequence
\begin{equation}\label{eq:Pi}
|\Pi(x)|\le C\beta|x|^{6-3d}.
\end{equation}
To conclude from the definitions and estimate above that $G^{\saw}*\Delta^{\saw}=\delta_0$ note that, by \eqref{KJiden}, 
	\begin{align}
	\label{Gsaw-rec}
	G^{\saw}(x)
        &\stackrel{\textrm{(\ref{Gsaw})}}{=}
          \sum_{\gamma:0\to x}
          \lambda^{\len(\gamma)}K[0,\len(\gamma)](\gamma)
        \stackrel{\textrm{(\ref{KJiden})}}{=}
          \delta_0+
          \sum_{\len(\gamma)\ge 1}
          \lambda^{\len(\gamma)}\Big(K[1,\len(\gamma)]+
          \sum_{m=2}^{\len(\gamma)}J[0,m]K[m,\len(\gamma)]\Big)\nonumber\\
&=\delta_0+\lambda\sum_{y\colon \|y\|=1}G^{\saw}(x-y)+\sum_y \Pi(y)
          G^{\saw}(x-y),
	\end{align}
where the last equality is derived as follows: the $K[1,\len(\gamma)]$
terms we divide according to $\gamma(1)$, which we denote by
$y$. Translation invariance gives that each term is exactly $G^{\saw}(x-y)$. The terms containing $J$ are
divided according to $\gamma(m)$, which we denote by $y$, and again by
translation invariance the sum over $K$ gives $G^{\saw}(x-y)$. Finally, the change of order of summation is justified by (\ref{eq:Pi absolutely}) and $G^{\saw}(x)\le C|x|^{2-d}$. This explains
(\ref{Gsaw-rec}). Rearranging (\ref{Gsaw-rec}) gives
$G^{\saw}*\Delta^{\saw}=\delta_0$, as required.

We move to prove properties (1)-(3) of $\Delta^{\saw}$. The symmetry of $\Delta^{\saw}$ is immediate from the construction,
and the property that $\sum\Delta^{\saw}(x)\ge0$ comes from summing the relation $(G^{\saw}*\Delta^{\saw})(x)=\delta_{0}(x)$ over $x\in \Z^d$, which gives that 
	\begin{equation}\label{eq:muchi}
	\sum_{x\in\mathbb{Z}^{d}}\Delta^{\saw}(x)
	=\frac{1}{\sum_{x\in \Z^d} G^{\saw}(x)},
	\end{equation}
and the last term is clearly non-negative as well as finite since $\lambda<\lambda_c$ which means that $G^{\saw}(x)$ decays
exponentially as $x\to\infty$. Thus the only property
that needs verification is the bound for $\Delta^{\saw}-\Delta^{\rw}$.

Recall that we need to choose some $\lambda'$ and estimate $\Delta_{\lambda}^{\saw}-\Delta_{\lambda'}^{\rw}=\Delta_\lambda^{\rw}-\Delta_{\lambda'}^{\rw}-\Pi_{\lambda}$.
By (\ref{eq:Pi}) 
	\[
	1-2d\lambda=\sum_{x}\Delta_{\lambda}^{\rw}(x)
	=\sum_{x}\big(\Pi_{\lambda}(x)+\Delta^{\saw}_{\lambda}(x)\big)\ge-\sum_{x}C\beta|x|^{3(2-d)}=-C\beta.
	\]
which means that we can choose $\lambda'=\min(\lambda,\frac{1}{2d})$ to
satisfy the conditions of the lemma. Hence, the only thing left is to prove (\ref{eq:Pi absolutely}). 

We next rewrite \eqref{Pi} in a form that can be used to obtain good bounds on $\Pi(x)$. For this, 
we start by introducing the laces that give the lace expansion its name.
A {\em lace} is a minimally connected graph, i.e., a connected graph for which
the removal of any edge would result in a disconnected graph. The set of
laces on $[a,b]$ is denoted $\Lcal[a,b]$, and the set of laces on
$[a,b]$ consisting of exactly $N$ edges is denoted $\Lcal^{\sss (N)} [a,b]$.
Given a connected graph $\Gamma$, the following prescription associates to
$\Gamma$ a unique lace ${L}_\Gamma$:  The lace ${L}_\Gamma$ consists
of edges $a_1 b_1, a_2 b_2, \ldots$, with $a_1,b_1,b_2,a_2, \ldots$
determined, in that order, by
	\[
    	b_1 = \max \{t \colon a t\in \Gamma \} ,  \;\;\;\; a_1 = a,
	\]
	\[
    	b_{i} = \max \{t \colon \exists a < t_{i-1} \mbox{ such that }
    	at \in \Gamma  \}, \quad
    	a_{i} = \min \{ s \colon sb_{i} \in \Gamma \}  .
	\]
See Figure \ref{fig:abx}. Given a lace $L$, the set of all edges $st \not\in L$
such that ${L}_{L\cup \{st\} } = L $ is denoted  $\Ccal (L)$.
Edges in $\Ccal (L)$ are said to be {\em compatible} with $L$. Now,
$L_\Gamma=L$ if and only if $L\subset\Gamma$ and all edges in
$\Gamma\setminus L$ are compatible with $L$. This allows to write
\[
\sum_{\Gamma:L_\Gamma=L}\prod_{st\in \Gamma\setminus L}U_{st}
=\prod_{st\in\Ccal(L)}(1+U_{st})
\]
and then partially resum the right-hand side of \eqref{Jconn}, to obtain
	\eqan{
	\label{Jconngraph}
    	J[a,b] & =
    	\sum_{L \in \Lcal[a,b] } \; \sum_{\Gamma\colon {L}_\Gamma = L}
    	\; \prod_{st \in L} U_{st} \prod_{s't' \in \Gamma \backslash L}
    	U_{s't'}
    	= 
    	\sum_{L \in \Lcal[a,b] } \prod_{st \in L} U_{st}
    	\prod_{s't' \in \Ccal (L) } ( 1 + U_{s't'} ) .
	}
For $0 \leq a<b$, we define $J^{\sss(N)}[a,b]$ to be the contribution to
\eqref{Jconngraph} coming from laces consisting of exactly $N$ edges:
	\eqn{
	\label{JNdef}
    	J^{\sss (N)} [a,b] =
    	\sum_{L \in \Lcal^{(N)} [a,b] } \prod_{st \in L} (-U_{st})
    	\prod_{s't' \in \Ccal (L) } ( 1 + U_{s't'} ), \qquad N \geq 1.
	}
Then, by \eqref{Pi},
	\eqn{
	\label{Jdef}
    	J [a,b] = \sum_{N=1}^{\infty} (-1)^N J^{\sss (N)} [a,b]
	\qquad
	\text{and}
	\qquad
    	\Pi(x) = \sum_{N=1}^\infty (-1)^N \Pi^{\sss(N)} (x) ,
	}
where we define
	\eqan{
	\label{PiNdef}
    	\Pi^{\sss(N)} (x) & =
    	\sum_{\gamma\colon  0 \rightarrow x} \lambda^{\len(\gamma)}J^{\sss(N)} [0,\len(\gamma)](\gamma)
    	\\ \nonumber
    	&=
   	\sum_{\gamma\colon  0 \rightarrow x} \lambda^{\len(\gamma)}    \sum_{L \in \Lcal^{\sss(N)} [0,\len(\gamma)] } \prod_{st \in L}
    	(-U_{st}(\gamma))
    	\prod_{s't' \in \Ccal (L) } ( 1 + U_{s't'}(\gamma) ).
	}

\begin{figure}
\centering\input{lacepdf.tex}

\caption{Laces}\label{fig:abx}
\end{figure}
We will now show that the sum over $N$ converges \emph{absolutely}.
The product over $st\in\mathcal{C}(L)$ will be easier to handle
when we restrict it. Let therefore $\mathcal{D}(L)$ be the set of edges
$st$ such that the open interval $(s,t)$ does not contain an $a_i$ or $b_i$ for any $(a_i,b_i)\in L$.
Clearly $\mathcal{D}(L)\subseteq\mathcal{C}(L)$ and therefore $\prod_{\mathcal{C}(L)}(1+U_{st})\le\prod_{\mathcal{D}(L)}(1+U_{st})$.
Once we restrict, the sum over $\gamma$ becomes independent between
any two consecutive elements of $L$. Here we call the ordered set $\{a_{1},a_2, b_{1},a_3\dotsc,b_{\sss N-1},b_{\sss N}\}$ the elements of the lace $L=\{a_{1}b_{1},\dotsc,a_{\sss N}b_{\sss N}\}$.

Calling $\gamma_{i}$ the piece of the path $\gamma$ between the $i^{\textrm{th}}$ and $(i+1)^{\textrm{st}}$ elements of $L$, we get 
	\[
	\prod_{st\in\mathcal{D}(L)}(1+U_{st}(\gamma))=\prod_{i=1}^{|L|}W(\gamma_{i})
	\]
We now claim that inserting this into the definition of $\Pi^{(N)}$ gives, for $N>1$,
	\begin{multline}
	|\Pi^{\sss(N)}(x)|\le\beta^N\sum_{0=x_{1},\dotsc,x_{{\sss N}}=x}G^{\saw}(x_{1}-x_{2})^{2}
	G^{\saw}(x_{3}-x_{1})
	G^{\saw}(x_{2}-x_{3})\times \\
	\dotsb \times G^{\saw}(x_{{\sss N}-1}-x_{{\sss N}-2}) 
	G^{\saw}(x_{{\sss N}}-x_{{\sss N}-2}) G^{\saw}(x_{{\sss N}}-x_{{\sss N}-1})^2.
	\label{eq:boundAl}
\end{multline}
(see again Figure 1). Indeed, the terms $U_{st}$ give the factor $\beta^N$ as well as restrictions $\gamma(a_i)=\gamma(b_i)$ for all $i$. Under this restrictions $\gamma$ breaks into paths $\gamma_i$ which are independent given their endpoints, so their sum gives $G^{\saw}$. This justifies (\ref{eq:boundAl}).

This description does not quite hold for $N=1$, as in this case
we do not get $G^{\saw}(0)$ as expected, since we are missing the term
$(1+U_{0n}(\gamma))$ in the product, but we may still bound 
	\[
	\Pi^{\sss(1)}(0)\le\frac{\beta}{1-\beta}G^{\saw}(0),\qquad\qquad \Pi^{\sss(1)}(x)=0\quad\forall x\ne0.
	\]
With these estimates in hand, we can bound $\Pi$.

Now $\Pi^{\sss(1)}$ clearly poses no problems. For $\Pi^{\sss(2)}$ we have $|\Pi^{\sss(2)}(x)|\le\beta^{2}G^{\saw}(x)^{3}.$
By our assumptions $G^{\saw}(x)^{3}\le27G^{\rw}(x)^{3}\le C|x|^{6-3d}$, as required.
For the next terms we need the following lemma:

\begin{lem}
\label{lem:justacalcul}
Let $d>4$. For any $u,v\in\mathbb{Z}^{d}$,
	\[
	\sum_{w\in\mathbb{Z}^{d}}|w|^{4-2d}|w-u|^{2-d}|w-v|^{2-d}\le C|u|^{2-d}|v|^{2-d}.
	\]
\end{lem}
\begin{proof}
By Cauchy-Schwarz
	\eqn{
	\Big(\sum_{w\in\mathbb{Z}^{d}}|w|^{4-2d}|w-u|^{2-d}|w-v|^{2-d}\Big)^2
	\le \Big(\sum_{w\in\mathbb{Z}^{d}}|w|^{4-2d}|w-u|^{4-2d}\Big)
	\Big(\sum_{w\in\mathbb{Z}^{d}}|w|^{4-2d}|w-v|^{4-2d}\Big).
	}
For $d>4$, each term can be estimated simply by splitting the sum to $|w|>|u|/2$ and $|w|\le |u|/2$, see a detailed calculation in \cite[Proposition 1.7(i)]{HHS03}. We get
	\[
	\sum_{w\in\mathbb{Z}^{d}}|w|^{4-2d}|w-u|^{4-2d}\leq C|u|^{4-2d}.\qedhere
	\]
\end{proof}	
To use this lemma, define $A^{\sss(2)}(x)=|x|^{6-3d}$, and for $N\ge 3$,
	\[
	A^{\sss(N)}(x)=\sum_{0=x_{1},\dotsc,x_{{\sss N}}=x}
	|x_{1}-x_{2}|^{4-2d}|x_{3}-x_{1}|^{2-d}|x_{2}-x_{3}|^{2-d}
	\dotsb|x_{{\sss N}}-x_{{\sss N}-1}|^{4-2d}
	\]
(the terms taken from Figure \ref{fig:abx}) so that by (\ref{eq:boundAl}) and $G^{\saw}(x)\le3G^{\rw}(x)\le C|x|^{2-d}$
we get $|\Pi^{\sss(N)}(x)|\le(C\beta)^{N}A^{\sss(N)}(x)$. 

We will show by induction that
	\eqn{
	A^{\sss(N)}(x)\le C^N |x|^{6-3d},
	}
which will show that $|\Pi^{\sss(N)}(x)|\le(C\beta)^{N}A^{\sss(N)}(x)\leq (C_1\beta)^{N}|x|^{6-3d},$ 
as required.

There is nothing to prove for $N=2$. To advance the induction hypothesis, we write
	\eqan{
	A^{\sss(N+1)}(x)
 	& =\sum_{x_{2},\dotsc,x_{{\sss N}-1}}\bigg(\textrm{terms without }x_{{\sss N}}\bigg)
	\sum_{x_{\sss N}}|x-x_{{\sss N}}|^{4-2d}|x_{{\sss N}}-x_{{\sss N}-1}|^{2-d}|x_{{\sss N}}-x_{{\sss N}-2}|^{2-d}\nn\\
 	& \le C\sum_{x_{2},\dotsc,x_{{\sss N}-1}}\bigg(\textrm{terms without }x_{{\sss N}}\bigg)|x-x_{{\sss N}-1}|^{2-d}|x-x_{{\sss N}-2}|^{2-d}
	=CA^{\sss (N)}(x),
	}
where the inequality follows from using Lemma \ref{lem:justacalcul}
with $w=x-x_{{\sss N}}$, $u=x-x_{{\sss N}-1}$ and $v=x-x_{{\sss N}-2}$ (the
``terms without $x_{\sss N}$'' contain one copy of $|x-x_{N-1}|^{2-d}$,
and with the second copy from Lemma \ref{lem:justacalcul} we get the
correct power, $4-2d$).

We may now choose $\beta_{0}$, and we choose it to be $1/(2C_{1})$.
With this choice of $\beta_{0}$, for every $\beta<\beta_{0}$, 
$\Pi^{\sss (N)}(x)$ decays exponentially with $N$, showing
the estimate $|\Pi(x)|\le C\beta|x|^{3(d-2)}$ and completing the proof of Lemma \ref{lem:lace}.
\end{proof}

\section{Proof of Lemma \ref{lem:solve}, with Banach algebras}

We start by defining a norm on $f\colon \mathbb{Z}^{d}\to\mathbb{R}$ by
	\[
	\|f\|:=\max\Big\{\sum_{x\in\mathbb{Z}^{d}}|f(x)|,
	\sup_{x\in\mathbb{Z}^{d}}|f(x)|\cdot|x|^{d}\Big\}
	\]
where $|x|$ denotes, say, the $\ell^{2}$ norm in $\Z^d$. Our norm is a Banach
algebra norm with respect to convolution, up to a constant. Indeed,
let $f$ and $g$ satisfy that $\|f\|,\|g\|\le1$. Then
	\begin{equation}
	\sum_{x\in\mathbb{Z}^{d}}|(f*g)(x)|\le\Big(\sum_{x}|f(x)|\Big)
	\Big(\sum_{x}|g(x)|\Big)
	\le\|f\|\cdot\|g\|\le1,\label{eq:fgl1}
	\end{equation}
and for every $x\in\mathbb{Z}^{d}$,
	\[
	|(f*g)(x)|\le\sum_{y\in\mathbb{Z}^{d}}|f(y)| |g(x-y)|
	=\sum_{|y|>|x-y|}|f(y)| |g(x-y)|+\sum_{|y|\le|x-y|}|f(y)| |g(x-y)|.
	\]
For the first term, whenever $|y|>|x-y|$, we have $|y|>|x|/2$
and hence
	\[
	\sum_{|y|>|x-y|}|f(y)\|g(x-y)|
	\le\sup_{|y|>\frac{1}{2}|x|}|f(y)|\cdot\sum_{|y|>|x-y|}|g(x-y)|
	\le\left(\tfrac{1}{2}|x|\right)^{-d}\|f\|\cdot\|g\|.
	\]
A similar estimate holds for the other term, now using that $|x-y|\geq |x|/2$ when $|y|\le|x-y|$, and we get
	\[
	|(f*g)(x)|\le2^{d+1}|x|^{-d}.
	\]
With (\ref{eq:fgl1}) we get
	\[
	\|f*g\|\le2^{d+1}\|f\|\cdot\|g\|.
	\]
In particular $B=\{f\colon \|f\|<\infty\}$ has a Banach algebra structure.
While one can find an equivalent norm on $B$ that is a proper Banach
algebra norm, it will be simpler to just use the norm defined above.
We get that if $\|f-\delta_{0}\|<2^{-d-1}$, then $f$ is invertible
in the algebra and 
	\begin{equation}
	\|f^{-1}-\delta_{0}\|\le\frac{2^{d+1}\|f-\delta_{0}\|}
	{1-2^{d+1}\|f-\delta_{0}\|}.
	\label{eq:inversion}
	\end{equation}

The following lemma forms the heart of our analysis:

\begin{lem}
\label{lem:Qrho} Fix $d>2$. Assume $\rho\colon \mathbb{Z}^{d}\to\mathbb{R}$ satisfies 
\begin{enumerate}
\item $\rho$ is symmetric to coordinate permutations and flipping.
\item $\sum_{x}\rho(x)=0$.
\item $|\rho(x)|\le|x|^{-d-4}$.
\end{enumerate}
Then $\|\rho*G^{\rw}\|\le C$.
\end{lem}

\begin{proof}
By \cite{U98}, the random walk Green's function has an expansion of the form 
	\begin{equation}
	G^{\rw}(x)=a|x|^{2-d}+b|x|^{-d}+O(|x|^{-d-2}). 
	\label{Edgeworth-Grw}
	\end{equation}%
(such an expansion is sometimes called an \textquotedblleft Edgeworth
expansion\textquotedblright ). Therefore, the sum defining $(\rho \ast G^{\rw%
})(x)$ converges absolutely for every $x\in \mathbb{Z}^{d}$ so this function
is well defined. Write 
	\begin{equation*}
	(\rho \ast G^{\rw})(x)=\sum_{y}\rho (y)G^{\rw}(x-y)=\sum_{|y|<\frac{1}{2}%
	|x|}\rho (y)G^{\rw}(x-y)+\sum_{|y|\geq \frac{1}{2}|x|}\rho (y)G^{\rw%
	}(x-y)=I+II.
	\end{equation*}%
We start with $I$ and write it as $I=I_{1}+I_{2}+I_{3}$, $I_{j}$ from the
three parts on the RHS of (\ref{Edgeworth-Grw}). For $I_{1}
$ we Taylor expand $|x-y|^{2-d}$ around $x$ to order 3 and get 
	\begin{align*}
	\frac{1}{a}I_{1}& =\sum_{|y|<\frac{1}{2}|x|}\rho (y)|x-y|^{2-d}=\sum_{|y|<%
	\frac{1}{2}|x|}\rho (y)\bigg[|x|^{2-d}+\sum_{i=1}^{d}y_{i}\cdot
	(2-d)x_{i}|x|^{-d} \\
	& \quad +\sum_{i,j=1}^{d}y_{i}y_{j}\left( -(2-d)d\cdot
	x_{i}x_{j}|x|^{-d-2}+\delta _{ij}(2-d)|x|^{-d}\right) +O(|y|^{3}|x|^{-d-1})%
	\bigg].
	\end{align*}%
We now bound these terms. For the first we write 
	\begin{equation*}
	\Big|\sum_{|y|<\frac{1}{2}|x|}\rho (y)|x|^{2-d}\Big|=|x|^{2-d}\Big|%
	\sum_{|y|\geq \frac{1}{2}|x|}\rho (y)\Big|\leq |x|^{2-d}\sum_{|y|\geq \frac{1%
	}{2}|x|}|y|^{-d-4}\leq C|x|^{-d-2},
	\end{equation*}%
where in the equality we used that $\sum_{x}\rho (x)=0$. For the second, we
use the symmetry of $\rho $ to flipping of $y_{i}$ to conclude that 
	\begin{equation*}
	\sum_{|y|<\frac{1}{2}|x|}\rho (y)y_{i}=0
	\end{equation*}%
and similarly for the off-diagonal second-order terms, i.e., for $\smash{\sum\limits_{|y|<\frac{1}{2}|x|}} \rho
(y)y_{i}y_{j}$ for $i\neq j$. The on-diagonal terms are equal to 
	\begin{equation*}
	\sum_{i=1}^{d}\sum_{|y|<\frac{1}{2}|x|}\rho (y)y_{i}^{2}\left(
	(2-d)|x|^{-d}-(2-d)d\cdot x_{i}^{2}|x|^{-d-2}\right) 
	\end{equation*}%
and the symmetry of $\rho $ to coordinate permutations shows that $\smash{\sum\limits_{|y|<\frac{1}{2}|x|}} \rho
(y)y_{i}^{2}$ does not depend on $i$. We take it out of the sum and see that
	\begin{equation}
	\sum_{i=1}^{d}\left( (2-d)|x|^{-d}-(2-d)d\cdot x_{i}^{2}|x|^{-d-2}\right) =0.
	\label{harmonicity}
	\end{equation}%
Finally, the third order terms are bounded by 
	\begin{equation*}
	\sum_{|y|<\frac{1}{2}|x|}|\rho (y)||y|^{3}|x|^{-d-1}\leq |x|^{-d-1}\sum_{|y|<%
	\frac{1}{2}|x|}|y|^{-d-4}|y|^{3}\leq C|x|^{-d-1}.
	\end{equation*}%
Putting all these estimates together gives that 
	\begin{equation*}
	|I_{1}|\leq C|x|^{-d-1}.
	\end{equation*}%
The estimates of $I_2$, $I_3$ and $II$ are much simpler. To estimate
$|I_{2}|,$ we Taylor expand $|x-y|^{-d}$ to first order, i.e.\ $|x-y|^{-d}=|x|^{-d}+O(|y||x|^{-d-1})$. A
similar argument shows that $|I_{2}|\leq C|x|^{-d-1}$. (We don't need the
harmonicity of $|x|^{2-d}$ which is the true reason for the
cancellation in (\ref{harmonicity}) above.) For $I_3$, we bound, again using that $|x-y|>\tfrac12 |x|$ when $|y|<\tfrac12 |x|$,
	\begin{equation*}
	|I_{3}|\leq C\sum_{|y|<\frac{1}{2}|x|}|\rho (y)||x-y|^{-d-2}
	\le C|x|^{-d-2}\sum_{y\in\Z^d}|\rho(x)|\leq C|x|^{-d-2}.
	\end{equation*}%
For $II$, we split $II=II_{1}+II_{2}$, depending on whether $|x-y|\leq |x|$ or $|x-y|>|x|$ and use $|y|\geq \tfrac12 |x|$ and $|\rho(y)|\leq |y|^{-d-4}$ to bound
	\begin{equation}
	|II_{1}|\leq 2^{d+4}|x|^{-d-4}\sum_{y\colon |x-y|\leq |x|}G^{\rw}(x-y)\leq
	C|x|^{-d-2},
	\end{equation}%
while  for use $|x-y|> |x|$, we use that $G^{\rw}(x)\leq C|x|^{2-d}$ by \eqref{Edgeworth-Grw} to bound
	\begin{equation}
	|II_{2}|\leq C|x|^{2-d}\sum_{|y|\geq \frac{1}{2}|x|}|\rho (y)|\leq
	C|x|^{2-d}|x|^{-4}=C|x|^{-d-2}.
	\end{equation}%
We conclude that $|(\rho \ast G^{\rw})(x)|\leq
C|x|^{-d-1}$. This proves the lemma.
\end{proof}

\noindent
{\bf Remark.} It seems as if Lemma \ref{lem:Qrho} makes a stringent requirement on
the types of random walks for which the argument can be
applied, as Edgeworth expansions are not easy to get. For example, if
one wishes to apply the argument for the Cayley graph of, say, the
Heisenberg group, then the natural analog of an Edgeworth expansion is
not known. We have a more roundabout proof of Lemma \ref{lem:Qrho}
that only uses the local central limit theorem. 
This argument will be presented elsewhere.
\smallskip

%

In the following lemma, we extend Lemma \ref{lem:Qrho} to the subcritical SRW Green's function $G_{\mu}^{\rw}$,  i.e., $G_{\mu}^{\rw}(x)=\sum_{n\geq 0}(2d\mu)^{n}p_{n}(x)$. Note that $G_{\mu}^{\rw}*\Delta_{\mu}^{\rw}=\delta_{0}$.
\begin{lem}
\label{lem:rhoG} 
Let $d>2$. Let $\rho$ be as in Lemma \ref{lem:Qrho} and let
$\mu\in\left[-\frac 1{4d},\frac{1}{2d}\right]$. Then 
	\[
	\|\rho*G_{\mu}^{\rw}\|\le C,
	\]
where $C$ does not depend on $\mu$.
\end{lem}

\begin{proof} 
Write
	\[
	\rho*G_{\mu}^{\rw}=\rho*G^{\rw}*\Delta^{\rw}*G_{\mu}^{\rw}.
	\]
(recall that $G^{\rw}=G^{\rw}_{1/(2d)}$ and $\Delta^{\rw}=\Delta^{\rw}_{1/(2d)}$). Notice that we do not need to put any parenthesis in this expression
as associativity follows from the fact that all sums converge absolutely,
which can be easily seen from the upper bounds for the various terms.
Since we already know that $\|\rho*G^{\rw}\|\le C$ by Lemma \ref{lem:Qrho},
we need only bound $\|G_{\mu}^{\rw}*\Delta^{\rw}\|$ (note again that
this is not $\Delta_{\mu}^{\rw}$ but rather $\Delta_{1/(2d)}^{\rw}$). 
Noting that $p_{n}*p_{m}=p_{n+m}$, we get
	\[
	G_{\mu}^{\rw}*\Delta^{\rw}
	=\Big(\sum_{n=0}^{\infty}(2d\mu)^{n}p_{n}\Big)*(\delta_{0}-p_{1})
	=\delta_{0}-(1-2d\mu)\sum_{n=1}^{\infty}(2d\mu)^{n-1}p_{n}.
	\]
For $\mu=1/(2d)$, this is identically equal to $\delta_0$, and there is nothing to prove. Thus, we can assume that $\mu\in [-\frac{1}{4d}, \frac{1}{2d})$.
Now, since $p_{n}(x)\le Cn^{-d/2}\e^{-c|x|^{2}/n}$ we get that $\|p_{n}\|\le C$.
Hence,
	\[
	\|G_{\mu}^{\rw}*\Delta^{\rw}\|\le1+(1-2d\mu)\sum_{n=1}^{\infty}(2d\mu)^{n-1}\|p_{n}\|\le C,
	\]
and the lemma is proved. 
\end{proof}

\begin{proof}[Proof of Lemma \ref{lem:solve}]
Recall that the input of the lemma
is a function $\Delta$ satisfying conditions (1)-(3) of Lemma \ref{lem:lace}.
These conditions are quite close to the conditions on $\rho$
in Lemmas \ref{lem:Qrho} and \ref{lem:rhoG}, only a linear map is required to pass from
one to the other. We define $\mu$ to be such that 
	\eqn{
	\label{mu-choice}
	\sum_{x\in\mathbb{Z}^{d}}(\Delta-\Delta_{\mu}^{\rw})(x)=0.
	}
(This choice is closely related to the choice of constants $\lambda, \mu$ in \cite[(2.29)]{HHS03}.)
To use Lemma \ref{lem:rhoG} we need to justify why $\mu\in[-\frac 1{4d},\frac{1}{2d}]$. Since
$\sum_x\Delta_\mu^{\rw}(x)=1-2d\mu$, we get
	\[
	\mu=\frac{1}{2d}\Big(1-\sum_{x\in\Z^d}\Delta(x)\Big).
	\]
The upper bound
$\mu\le\frac{1}{2d}$ is automatic since $\sum_x \Delta(x)\ge 0$. For the
lower bound, we need to show that $\sum_x\Delta(x)\le \nicefrac 32$. This follows because $\sum_x \Delta^{\rw}_{\lambda'}(x)\le 1$ with $\lambda'$ chosen as in Lemma \ref{lem:lace}(3), and $\sum_x |\Delta_\lambda(x)-\Delta^{\rw}_{\lambda'}(x)|\leq C\beta$ again by Lemma \ref{lem:lace}(3), so for
$\beta$ sufficiently small we will have $\mu\ge -\frac 1{4d}$, as needed.

Next we note that 
	\begin{equation}\label{eq:munotlambda}
	\sum |\Delta(x)-\Delta^{\rw}_\mu(x)|\le C_2\beta|x|^{-d-4}.
	\end{equation}
Indeed, at every $x$ that is not a neighbour of $0$ this is an immediate corollary from our condition (3) of Lemma \ref{lem:lace}. For the neighbours, we note that 
	\eqn{
	\label{lambda'-mu-dif}
	|\lambda'-\mu|
	=\frac{1}{2d}\Big|\sum_{x\in\Z^d}(\Delta^{\rw}_{\lambda'}-\Delta^{\rw}_\mu)(x)\Big|
	\stackrel{\textrm{(\ref{mu-choice})}}{=}
	\frac{1}{2d}\Big|\sum_{x\in\Z^d}(\Delta^{\rw}_{\lambda'}-\Delta)(x)\Big|
	\le \frac{C\beta}{2d},
	}
where the last inequality is again from condition (3) of Lemma
\ref{lem:lace}. Thus, for $x$ a neighbour of the origin, we conclude that 
	\[
	\Delta(x)-\Delta^{\rw}_\mu(x)=\mu-\lambda'+\Delta^{\saw}_{\lambda}(x)-\Delta^{\rw}_{\lambda'}(x)
	=O(\beta),
	\]
by \eqref{lambda'-mu-dif} and condition (3) of Lemma
\ref{lem:lace}. This shows \eqref{eq:munotlambda}.

We next define 
	\[
	\rho=\frac{1}{C_{2}\beta}(\Delta-\Delta_{\mu}^{\rw}),
	\]
with $C_{2}$ being the constant from \eqref{eq:munotlambda}.
This $\rho$ has the required properties, so that, by Lemma \ref{lem:rhoG},
	\[
	\|\rho*G_{\mu}^{\rw}\|\le C.
	\]
In turn, this implies
	\[
	\|(\Delta-\Delta_{\mu}^{\rw})*G_{\mu}^{\rw}\|\le C\beta.
	\]
But this is exactly $\Delta*G_{\mu}^{\rw}-\delta_{0}$. This means
that $\Delta*G_{\mu}^{\rw}$ is invertible if $\beta$ is sufficiently
small (recall (\ref{eq:inversion})), and further that we have $(\Delta*G_{\mu}^{\rw})^{-1}=\delta_{0}+E$
with $\|E\|\le C\beta$. Our required function is now
	\[
	G=(\Delta*G_{\mu}^{\rw})^{-1}*G_{\mu}^{\rw},
	\]
which is clearly an inverse for $\Delta$. To see that $G(x)\le 2G^{\rw}(x)$
write
	\eqn{
	\label{G-E-rewrite}
	G=G_{\mu}^{\rw}+E*G_{\mu}^{\rw}.
	}
Since $G_{\mu}^{\rw}(x)\le G^{\rw}(x)$, because $\mu\leq 1/(2d)$, we need only estimate $E*G_{\mu}^{\rw}$.
We write
	\eqn{
	(E*G_{\mu}^{\rw})(x)=\sum_{y}G_{\mu}^{\rw}(y)E(x-y)
	=\sum_{|y|<\frac{1}{2}|x|}G_{\mu}^{\rw}(y)E(x-y)
	+\sum_{|y|\ge\frac{1}{2}|x|}G_{\mu}^{\rw}(y)E(x-y)=I+II.
	}
For $I$ we use that $|y|<\frac{1}{2}|x|$ implies that $|x-y|\ge\frac{1}{2}|x|$
so $|E(x-y)|\le C\beta|x|^{-d}$ and hence
	\eqn{
	|I|\le C\beta|x|^{-d}\sum_{|y|<\frac{1}{2}|x|}|G_{\mu}^{\rw}(y)|
	\le C\beta|x|^{-d}\sum_{|y|<\frac{1}{2}|x|}|y|^{2-d}\le C\beta|x|^{2-d}.
	}
For $II$ we have
	\eqn{
	\label{E-last-bd}
	|II|\le\Big(\max_{|y|\ge\frac{1}{2}|x|}|G_{\mu}^{\rw}(y)|\Big)\cdot\sum_{y}|E(x-y)|
	\le C\left(\tfrac{1}{2}|x|\right)^{2-d}\cdot C\beta=C\beta|x|^{2-d}.
	}
We get that $|(E*G_{\mu}^{\rw})|\le C\beta|x|^{2-d}$, which means
that for $\beta$ sufficiently small, it is less than $G^{\rw}$.
This shows that $G(x)\le 2G^{\rw}(x)$, and thus completes the proof of Lemma \ref{lem:solve}.
\end{proof}

\subsection*{Remarks}
\label{page:rmrks}
\lazyenum
\item
Examining \eqref{G-E-rewrite}--\eqref{E-last-bd} in the proof of the last lemma shows that in fact
  we got that
	\[
	G^{\saw}_\lambda(x)=G^{\rw}_\mu(x)+O(\beta|x|^{2-d}).
	\]
Together with \eqref{Edgeworth-Grw}, this would prove \eqref{Green-asymp}, if only we could show that 
$\mu=\frac 1{2d}$ for $\lambda_c$, the critical $\lambda$. This is a classical fact, let us
sketch its proof for the convenience of the reader. Since
$\mu=\frac{1}{2d}\big(1-\sum_x\Delta(x)\big)$, it is equivalent
to showing that $\chi(\lambda)\to\infty$ as $\lambda\nearrow\lambda_c$
where $\chi(\lambda)=\sum_xG^{\saw}_\lambda(x)$ (recall (\ref{eq:muchi})). 
Let $c_n(x)=\smash{\sum\limits_{\substack{\gamma \colon 0\to x\\ \len(\gamma)=n}}}W^{\beta}(\gamma)$. Then, $c_n$ is submultiplicative in the sense that
	\[
	nc_n(x)\leq 2d\sum_{m=0}^{n-1} (c_m*c_{n-1-m})(x).
	\]
Entering this inequality into the definition of $G^{\saw}$ gives $\frac{\partial \chi(\lambda)}{\partial \lambda}\le 2d\chi(\lambda)^2.$
This shows that any point where $\chi(\lambda)<\infty$ must be
strictly subcritical, showing that $\chi(\lambda_c)=\infty$, as
needed. (The reader who finds this sketch too dense may see more
details in, say, \cite[Theorem 2.3]{S04}).

\item The result of the theorem is known as an ``infrared bound''. It implies the finiteness of the so-called \emph{bubble diagram}, which in turn implies various critical exponents. See again \cite[Theorem 2.3]{S04}.

\item Let us remark on the exponent $-d-4$ appearing in the inequality
$|\Delta^{\rw}-\Delta^{\saw}|\le C\beta|x|^{-d-4}$ of Lemma
\ref{lem:lace}. On the one hand, Lemma \ref{lem:lace} in fact gives a stronger bound with exponent $-3(2-d)$, see \eqref{eq:Pi}.
On the other hand, most of the proof of Lemma \ref{lem:solve} actually needs less, $|x|^{-d-2-\varepsilon}$ would have been enough. The
only place where the stronger estimate $|x|^{-d-4}$ is used is in
Lemma \ref{lem:rhoG}, in order to justify the associativity of the
convolution in the expression $\rho *
G^{\rw}*\Delta^{\rw}*G^{\rw}_\mu$. There are certainly ways to justify
associativity at that point under the weaker assumption $|\rho(x)|\le
|x|^{-d-2-\varepsilon}$, but an additional argument would be needed.
\end{enumerate}

\paragraph{{\bf Acknowledgements.}}
The work of RvdH is supported by the Netherlands
Organisation for Scientific Research (NWO) through VICI grant 639.033.806 
and the Gravitation {\sc Networks} grant 024.002.003.
The work of EB is supported by SNF grant
\verb!200020_138141!. GK is supported by the Israel Science Foundation
and the Jesselson Foundation, and by the CNRS during his visit to the
Institut Heni Poincar\'e.
This work was performed in part during a visit of RvdH to the Weizmann Institute, 
and when the authors met in Eurandom, Oberwolfach and the Institut Henri Poincar\'e. 
We thank these institutions for their hospitality.

\end{document}